\documentclass[12pt]{amsart}
\usepackage{amsmath}
\usepackage{amssymb,latexsym}

\newtheorem{prop}{Proposition}

\theoremstyle{remark}
\newtheorem*{cor}{\bf Corollary}
\newtheorem*{example}{Example}

\theoremstyle{definition}

\begin{document}

\pagestyle{plain}

\title{Compositions of Graphs Revisited} \author{Aminul Huq}

\email{aminul@brandeis.edu}
\urladdr{people.brandeis.edu/$\sim$aminul}

\maketitle
\begin{abstract}

The idea of graph compositions, which was introduced by A.
Knopfmacher and M. E. Mays, generalizes both ordinary compositions
of positive integers and partitions of finite sets. In their
original paper they developed formulas, generating functions, and
recurrence relations for composition counting functions for
several families of graphs. Here we show that some of the results
involving compositions of bipartite graphs can be derived more
easily using exponential generating functions.

\end{abstract}



\section{Introduction}
A composition of a graph $G$ is a partition of the vertex set of
$G$ into vertex sets of connected induced subgraphs of $G$.
Knopfmacher and Mays \cite{aa} found an explicit formula for
$C(K_{m,n})$, the number of compositions of the complete bipartite
graph $K_{m,n}$ in the form
\begin{equation}
C(K_{m,n})= \sum_{i=1}^{m+1} a_{m,i} i^n.
\end{equation}
where $(a_{i,j})$ is an array defined via the recurrences
$a_{m,0}=0$ for any nonnegative integer $m$, $a_{0,1}=1$,
$a_{0,n}=0$ for any $n>1$, and otherwise
\[a_{m,n}=\sum_{i=0}^{m-1}\binom{m-1}{i} a_{m-1-i,n-1}-\sum_{i=1}^{m-1}\binom{m-1}i a_{m-1-i,n-1}.\]

We will derive this result using exponential generating functions
and also show that we can express the coefficients $a_{m,i}$
explicitly in terms of the Stirling numbers of the second kind. We
first need to describe some basic properties of exponential
generating functions in two variables. We will use Stanley's
notation \cite{cc} throughout this paper.

\section{Exponential generating function in two
variables}

\begin{prop}\label{pr1} Given functions $ f,g : \mathbb{N}
\times \mathbb{N} \rightarrow K$, where $K$ is a field of
characteristic $0$, we define a new function $h:\mathbb{N} \times
\mathbb{N} \rightarrow K$ by
\[ h(\#X,\#Y)=\sum f(\#S,\#T)g(\#U,\#V) \] where $X$ and $Y$ are
finite sets and the sum is over all $S,T,U$ and $V$ such that $X=S
\uplus U$ and $Y=T \uplus V$; i.e., $X$ and $Y$ are disjoint
unions of $S,U$ and $T,V$ respectively.

Then
\begin{equation}\label{a-1}
E_h(x,y)=E_f(x,y)E_g(x,y),
\end{equation}
where the exponential generating function of $f$ is defined by
\[E_f(x,y)=\sum_{m,n=0}^{\infty} f(m,n)\dfrac{x^m}{m!}
\dfrac{y^n}{n!}.\]
\end{prop}

\begin{proof} If $\#X=m$ and $\#Y=n$, then there are
$\binom{m}{i}$ pairs $(S,U)$ and $\binom{n}{j}$ pairs $(T,V)$,
where $X=S \uplus U$ and $Y=T \uplus V$ with $\#S=i, \#T=j,
\#U=m-i$ and $\#V=n-j$. Then $h(m,n)$ is given by
\[h(m,n)=\sum_{i,j=0}^{m,n} \binom{m}{i} \binom{n}{j} f(i,j)
g(m-i,n-j)\] and this is equivalent to \eqref{a-1}.
\end{proof}

\begin{cor}\label{co1} Given functions $f_1,\dots,f_k :\mathbb{N}
\times \mathbb{N} \rightarrow K$, we can define a new function
$h:\mathbb{N} \times \mathbb{N} \rightarrow K$ by
\[ h(\#X,\#Y)=\sum \prod_{i=1}^{k} f_i(\#S_i,\#T_i) \] where the
sum is over all $(S_1,\dots,S_k)$ and $(T_1,\dots,T_k)$ such that
$X=\uplus_{i=1}^{k} S_i$ and $Y=\uplus_{i=1}^{k} T_i$. Then
\begin{equation}\label{a-2} E_h(x,y)=\prod_{i=1}^{k} E_{f_i}(x,y).
\end{equation}
\end{cor}

\begin{prop}\label{pr2} Given the function $ f : \mathbb{N} \times
\mathbb{N} \rightarrow K$, where $K$ is a field of characteristic
$0$ and $f(0,0)=0$, define a new function $h:\mathbb{N} \times
\mathbb{N} \rightarrow K$ such that for disjoint finite sets $X$
and $Y$,
\begin{equation}\label{a-3} h(\#X,\#Y)=\sum_{{\{S_1,\dots,S_k\}}}\prod_{i=1}^{k} f(\#(S_i\cap
X),\#(S_i\cap Y)). \end{equation} where the sum is taken over all
partitions $\{S_1,\dots,S_k\}$ of the set $X \cup Y$. Then
\[E_h(x,y)=\exp(E_f(x,y)).\]
\end{prop}

\begin{proof} Let $k$ be fixed. Then the blocks of the partition
$\{S_1, \dots ,S_k\}$ are all distinct and so are the pairs
${(S_i\cap X,S_i\cap Y)}$ for $i=1, \dots ,k$. So there are $k!$
ways of linearly ordering them. If we define $h_k(\#X,\#Y)$ by
\[ h_k(\#X,\#Y)=\sum_{{\{S_1,\dots,S_k\}}}\prod_{i=1}^{k} f(\#(S_i\cap
X),\#(S_i\cap Y))\] for a fixed value of $k$ then by the Corollary
to Proposition~\ref{pr1} we get
\[ E_{h_k} (x,y)= \dfrac{(E_f(x,y))^k}{k!} \]
Therefore summing over all $k \ge 0$ gives the desired result.
\end{proof}

\begin{example} Let $X$ and $Y$ be disjoint sets with $\#X=m$ and
$\#Y=n$ and let $C(m,n)$ be the number of connected bipartite
graphs between the sets $X$ and $Y$. Then
\begin{equation}
\label{ex} \exp\biggl(\sum_{m,n=0}^{\infty} C(m,n) \dfrac{x^m}{m!}
\dfrac{y^n}{n!}\biggr)=\sum_{m,n=0}^{\infty} 2^{mn}
\dfrac{x^m}{m!} \dfrac{y^n}{n!}.
\end{equation}

This can be seen easily because the coefficient of
$\dfrac{x^m}{m!} \dfrac{y^n}{n!}$ in the right-hand side of
\eqref{ex} is the number of bipartite graphs with bipartition
$(X,Y)$. On the other hand the number of such graphs in which the
vertex sets of the connected components are $\{S_1, S_2,
\dots,S_k\}$ is $\prod_{i=1}^{k} C(\#(S_i\cap X),\#(S_i\cap Y))$.
So summing over all partition $\{S_1, S_2, \dots,S_k\}$ of $X \cup
Y$ and applying Proposition \ref{pr2} shows that the number of
bipartite graphs with bipartition $(X,Y)$ is the coefficient of
$\dfrac{x^m}{m!} \dfrac{y^n}{n!}$ in the left-hand side.
\end{example}

\section{Compositions of bipartite graphs}

Let $G$ be a labelled graph with vertex set $V(G)$. A composition
of $G$ is a partition of $V(G)$ into vertex sets of connected
induced subgraphs of $G$. Thus a composition provides a set of
connected induced subgraphs of $G$, $\{G_1,G_2,\cdots,G_m\}$, with
the properties that $\bigcup_{i=1}^m V(G_i) = V(G)$ and for $i\ne
j, V(G_i)\bigcap V(G_j) =\emptyset$.

Let $C(G)$ denote the number of distinct compositions of the graph
$G$. For example, the complete bipartite graph $K_{2,3}$ has
exactly 34 compositions. In this section we will consider complete
bipartite graph only.

Consider a function $f:\mathbb{N} \times \mathbb{N} \rightarrow
\mathbb{Q}$ as follows: Given $m,n \in \mathbb{N}$ we define
\[f(m,n)=\begin{cases} 1, & \textrm{if $m>0$ and $n>0$}
\\ & \textrm{or $m=1$ and $n=0$} \\ & \textrm{or $m=0$ and $n=1
$}\\
  0, & \textrm{otherwise.}
\end{cases}\]
In other words $f(m,n)=1$ if $K_{m,n}$ is connected and $0$ if
$K_{m,n}$ is not connected. We also define $h:\mathbb{N} \times
\mathbb{N} \rightarrow \mathbb{Q}$ by
\[h(m,n)=\sum_{\{S_1,\dots,S_k\}} \prod_{i=1}^{k}f(\#(S_i \cap X),\#(S_i \cap
Y)).\] where $\#X=m$ and $\#Y=n$. Then $h(m,n)=C(K_{m,n})$ and
thus by Proposition \ref{pr2} we get
\begin{equation}
\label{e-01} \sum_{m,n=0}^{\infty} C(K_{m,n})
\dfrac{x^m}{m!}\dfrac{y^n}{n!}=\exp(E_f(x,y)).
\end{equation}
But from the definition of $E_f(x,y)$ we have
\begin{align*}
E_f(x,y)=\sum_{m,n=0}^{\infty} f(m,n) \dfrac{x^m}{m!}
\dfrac{y^n}{n!}
        &= \sum_{m,n=1}^{\infty} \dfrac{x^m}{m!}
\dfrac{y^n}{n!} +x+y\\
        &= \sum_{m=1}^{\infty}
\dfrac{x^m}{m!} \sum_{n=1}^{\infty} \dfrac{y^n}{n!}+x+y\\
        &= (e^x-1)(e^y-1)+x+y.
\end{align*}
So
\begin{equation}
 \sum_{m,n=0}^{\infty} C(K_{m,n})
\dfrac{x^m}{m!}\dfrac{y^n}{n!}=
e^{(e^{x}-1)(e^{y}-1)+x+y}.\label{e-1}
\end{equation}
Knopfmacher and Mays \cite{aa} showed that
\begin{equation}\label{e-2}
C(K_{m,n})= \sum_{i=1}^{m+1} a_{m,i} i^n.
\end{equation}
for some integers $a_{m,i}$. We will derive the same result here
from \eqref{e-1}.

We start by defining integers $a_{m,i}$ by
\begin{equation}
\label{e-4} \lambda e^{(e^{x}-1)(\lambda-1)+x} =
\sum_{m=0}^{\infty} \dfrac{x^m}{m!} \sum_{i=0}^{\infty} a_{m,i}
\lambda^i.
\end{equation}
Now equating the constant term in $x$ on both sides we get
\[\lambda=\sum_{i=0}^{\infty}a_{0,i}\lambda^i,\]
which shows that $a_{0,1}=1$ and $a_{0,i}=0$ for $i \neq 1$. We
observe that $a_{m,0}=0$. Now we equate the coefficients of
$\lambda^n$ on both sides of \eqref{e-4}.

On the left we have
\begin{align*}
[\lambda^n] \lambda e^{(e^{x}-1)\lambda} e^{x-(e^{x}-1)} &=
[\lambda^{n-1}] e^{(e^{x}-1)\lambda} e^{x-(e^{x}-1)}\\ &=
\dfrac{(e^x - 1)^{n-1}}{(n-1)!} e^{x-(e^{x}-1)}
\end{align*}

On the right we have
\[ [\lambda^n]\sum_{m=0}^{\infty} \dfrac{x^m}{m!} \sum_{i=0}^{\infty} a_{m,i}
\lambda^i=\sum_{m=0}^{\infty} \dfrac{x^m}{m!} a_{m,n}.\]

So
\[\dfrac{(e^x - 1)^{n-1}}{(n-1)!} e^{x-(e^{x}-1)}=\sum_{m=0}^{\infty} \dfrac{x^m}{m!} a_{m,n}.\]
Thus $a_{m,n}=0$ for $m<n-1$, i.e., $n>m+1$. So we may write
\eqref{e-4} as
\begin{equation}
\label{e-5} \lambda e^{(e^{x}-1)(\lambda-1)+x} =
\sum_{m=0}^{\infty} \dfrac{x^m}{m!} \sum_{i=1}^{m+1} a_{m,i}
\lambda^i.
\end{equation}
Letting $\lambda=e^y$ in equation \eqref{e-5} and using
\eqref{e-1} we get
\[ \sum_{m,n} C(K_{m,n}) \dfrac{x^m}{m!} \dfrac{y^n}{n!}
=\sum_{m=0}^{\infty} \dfrac{x^m}{m!} \sum_{i=1}^{m+1} a_{m,i}
e^{iy}.\] Equating coefficients of $\dfrac{x^m}{m!}
\dfrac{y^n}{n!}$ we get
\[C(K_{m,n})=\sum_{i=1}^{m+1} a_{m,i} i^n,\]
which is the desired result.

\section{Relation with Stirling numbers of the second kind}

The Stirling number of the second kind $S(n,m)$ counts the number
of ways of partitioning a set of $n$ elements into $m$ nonempty
sets. We can also get an expression for $a_{m,i}$ involving the
Stirling numbers of the second kind. To do this
let\[\rho_m(z)=\sum_{i=1}^{m+1} a_{m,i} z^i.\] Then setting
$\lambda=z+1$ in \eqref{e-5} we get
\[(1+z)e^x e^{(e^{x}-1)z}=\sum_{m=0}^{\infty} \dfrac{x^m}{m!} \rho_m(1+z).\]
or
\[e^x e^{(e^{x}-1)z}=\sum_{m=0}^{\infty}  \dfrac{x^m}{m!} \dfrac{\rho_m(z+1)}{z+1}.\]
Using the previous equation gives
\begin{equation}
\label{e-8} \dfrac{d }{dx}e^{(e^{x}-1)z}=z e^x
e^{(e^{x}-1)z}=\sum_{m=0}^{\infty}  \dfrac{x^m}{m!}
\dfrac{z}{z+1}\rho_m(z+1).
\end{equation}
The generating function for the Stirling numbers of the second
kind is
\begin{equation}
\label{e-9} \sum_{m,k} S(m,k) \dfrac{x^m}{m!} z^k=e^{(e^{x}-1)z}.
\end{equation}
So from \eqref{e-8} and \eqref{e-9} we get
\[\sum_{m,k} S(m+1,k) \dfrac{x^m}{m!} z^k= \dfrac{d }{dx} \sum_{m,k} S(m,k) \dfrac{x^m}{m!} z^k=
 \sum_{m=0}^{\infty} \dfrac{x^m}{m!} \dfrac{z}{z+1}\rho_m(z+1).\]
Equating coefficients of $\dfrac{x^m}{m!}$ we get
\[\sum_{k} S(m+1,k)  z^k= \dfrac{z}{z+1}\rho_m(z+1).\]
So
\[\rho_m(z)= z \sum_{k} S(m+1,k)
(z-1)^{k-1}.\]

{}From this we can easily extract the coefficient of $z^i$ to get

\[ a_{m,i}=\sum_k \binom{k-1}{i} (-1)^{k-i} S(m+1,k)\]

\section{Generalization} The generalization of \eqref{e-1} to
complete multipartite graphs is easy if we use the generating
function method. A complete multipartite graph is a multipartite
graph such that any two vertices that are not in the same part
have an edge connecting them. The number of edges for such graphs
are given by the formula $a_1 (a_2+ \cdots +a_n)+a_2 (a_3+\cdots
+a_n)+\cdots +a_{n-1} a_n$, where each ${a_i}$ is the number of
vertices in that part. If $K_{a_1,a_2,\dots,a_n}$ is a complete
multipartite graph with $a_1 + a_2 + \cdots +a_n $ vertices then
the number of compositions for this graph is given by the
generating function
\begin{equation}
 \sum_{a_1,a_2,\dots,a_n=0}^{\infty}
C(K_{a_1,a_2,...,a_n})
\dfrac{{x_1}^{a_1}}{a_1!}\dfrac{{x_2}^{a_2}}{a_2!}\cdots\dfrac{{x_n}^{a_n}}{a_n!}=
y_1 y_2 \cdots y_n e^{y_1 y_2 \dots y_n-y_1- y_2 \cdots -y_n +
n-1},
\end{equation}
where $y_i = e^{x_i}.$\\ \break

\end{document}